\documentclass{article}

\PassOptionsToPackage{round}{natbib}
\usepackage[preprint]{neurips_2026}

\usepackage[utf8]{inputenc} % allow utf-8 input
\usepackage[T1]{fontenc}    % use 8-bit T1 fonts
\usepackage[hyphens]{url}
\usepackage[hidelinks]{hyperref}     % hyperlinks
\usepackage{booktabs}       % professional-quality tables
\usepackage{nicefrac}       % compact symbols for 1/2, etc.
\usepackage{microtype}      % microtypography
\usepackage{xcolor}         % colors
\usepackage{amsmath,amssymb, amsthm, amsfonts}%
\usepackage{multirow}
\usepackage{mathtools}
\usepackage{makecell}
\usepackage{cellspace} 
\usepackage{bbm}
\newcommand{\R}{\mathbb{R}}

\newcommand{\E}{\mathbb{E}}
\newcommand{\1}{\mathbbm{1}}

\theoremstyle{definition}
\newtheorem{definition}{Definition}[section]
\theoremstyle{definition}

\newtheorem{assumptions}{Assumptions}[section]
\theoremstyle{plain}
\newtheorem{theorem}[definition]{Theorem}
\newtheorem{lemma}[definition]{Lemma} 
\newtheorem{proposition}[definition]{Proposition} 
\newtheorem{corollary}[definition]{Corollary}
\theoremstyle{definition}

\newtheorem{remark}[definition]{Remark}
\title{Robust and Fast Training via Per-Sample Clipping}

\author{%
  Davide Nobile \\
  Department of Mathematics\\
  University of Vienna\\
  \texttt{davide.nobile@univie.ac.at} \\
  \And
  Philipp Grohs \\
  Department of Mathematics\\
  University of Vienna\\
  and\\
  RICAM\\
  Austrian Academy of Sciences\\
   \texttt{philipp.grohs@univie.ac.at}
}

\begin{document}

\maketitle

\begin{abstract}
We propose a robust gradient estimator based on per-sample gradient clipping and analyze its properties both theoretically and empirically. We show that the resulting method, \textit{per-sample clipped SGD} (PS-Clip-SGD), achieves optimal in-expectation convergence rates for non-convex optimization problems under heavy-tailed gradient noise. Moreover, we establish high-probability convergence guarantees that match the in-expectation rates up to polylogarithmic factors in the failure probability.
We complement our theoretical results with multiple numerical experiments. In particular, we demonstrate that PS-Clip-SGD outperforms both vanilla SGD with momentum and standard gradient clipping when training AlexNet on the CIFAR-100 dataset, even after accounting for the additional computational time caused by per-sample clipping.
We also empirically show that, in the presence of gradient accumulation, applying clipping at the mini-batch level can improve training performance while incurring virtually no additional computational cost. This finding is particularly interesting, as it contradicts the common practice of applying clipping only after all accumulation steps have been completed.
\end{abstract}

\section{Introduction}

We consider the stochastic optimization problem
\begin{equation}\label{eq: objective function}
    \min_{x \in \mathbb{R}^d} \; F(x), \quad \text{where} \quad F(x) = \mathbb{E}_{\xi \sim \mathcal{D}}[f(x, \xi)] \,,
\end{equation}

where $F : \mathbb{R}^d \to \mathbb{R}$ is an $L$-smooth function and $\xi$ is a random variable drawn from an unknown distribution $\mathcal{D}$. This type of problems naturally appear in machine learning and statistics. By far the most popular and most analyzed method to tackle this type of optimization problems is Stochastic Gradient Descent (SGD). The convergence properties of SGD and its variants in this setting have extensively been studied under different assumptions on the problem ($\ref{eq: objective function}$). Under light tailed gradient noise, SGD has been shown to achieve optimal convergence rates for both convex and non-convex functions \citep{ghadimi_lan_2013_stochastic_nonconvex,arjevani_carmon_duchi_etal_2023_lower_bounds}. However, recent work \citep{pmlr-v97-simsekli19a,pmlr-v238-battash24a, zhang2020adaptive} has shown that in many machine learning applications, gradient noise might not posses bounded variance, making it necessary to analyze the problem under heavy tails conditions. We hereby mention that in the present work, by \textit{heavy-tailed} gradient noise, we mean that $|\nabla f(x,\xi)|$ has finite $p$-the moments for some $p \in (1,2]$, but may not have bounded variance.

While stochastic gradient descent (SGD) is the standard method under light-tailed noise assumptions, it can perform poorly when the gradient noise distribution lacks finite variance. This limitation has motivated the development of alternative methods designed to handle heavy-tailed noise. One widely used approach is \emph{clipped SGD}, which replaces the stochastic gradient estimator ${\nabla} f(x_t,\xi_t)$ with the clipped estimator
\begin{equation*}
    g(x_t,\xi_t) = \min\left\{1, \frac{\gamma_t}{|{\nabla} f(x_t, \xi_t)|} \right\} {\nabla} f(x_t, \xi_t) \,,
\end{equation*}
for a clipping threshold $\gamma_t > 0$. The resulting update rule is $x_{t+1} = x_t - \eta_t g(x_t,\xi_t)$.
Clipped SGD has been extensively studied, and both high-probability and in-expectation convergence guarantees have been established in the non-convex, heavy-tailed setting~\citep{zhang2020adaptive, improved-convergence}. However, these results do not achieve the optimal sample complexity lower bounds established by \citep{zhang2020adaptive}. Additionally, the high probability results require specific choices of $\log(1/\delta)$-dependent step size and clipping threshold to establish convergence with probability $1-\delta$.

An alternative approach is \emph{Normalized SGD}, originally proposed by \citep{nesterov1984minimization}. Recent work by \citep{hubler2024from} shows that Normalized SGD achieves optimal sample complexity in both expectation and high probability for non-convex problems with heavy-tailed noise. Additionally, Normalized SGD does not require any additional parameters to be tuned compared to vanilla SGD. However, the optimal step size used in \citep{hubler2024from} requires knowledge of the initialization gap $\Delta_1$. Furthermore, their analysis provides guarantees only for the sum of gradient norms, rather than the sum of squared gradient norms, which is a strictly stronger performance criterion due to Jensen's inequality.

Per-sample gradient clipping is a commonly used technique to ensure differentially private learning \citep{abadi_chu_goodfellow_etal_2016_dp_sgd}, where it is used to control the influence of any individual sample on the training process. In this context, its mathematical properties were for example analyzed in \citep{zhang_chen_hong_wu_yi_2022_clipping_fl, improv-ana-of-ps-clipping, xia_shen_yao_fu_xu_2023_adaptive_clipping}. Due to the computational cost of per-sample computations, its application and analysis has so far been limited to differential privacy and, in particular, is entirely unrelated to handling heavy-tailed noise. However, we believe that recent developments in the efficient computation of per-sample gradient norms \citep{lee_kifer_2021_fast_clipping_popets, ghost_clipping} make per-sample clipping a viable method for improving training performance in the presence of heavy-tailed noise. To the best of our knowledge, this is the first work to propose per-sample clipping as a method for improving training performance in this setting. 
Concretely, our goal is to develop a robust gradient estimator that provides strong convergence guarantees under heavy-tailed noise.
%The only works we are aware of that analyze PS-clipping do so in the context of differential privacy, for example \citep{zhang_chen_hong_wu_yi_2022_clipping_fl, improv-ana-of-ps-clipping, xia_shen_yao_fu_xu_2023_adaptive_clipping}. However, we emphasize that our work is completely unrelated to differential privacy. Rather, our goal is to develop a robust gradient estimator that provides strong convergence guarantees under heavy-tailed noise.

\subsection{Our contributions}

In this work, we propose and analyze a robust gradient estimator based on per-sample gradient clipping, and show that it achieves optimal convergence rates for non-convex problems with heavy-tailed gradient noise.

Specifically, we first show that our method, which we call \emph{per-sample clipped SGD} (PS-Clip-SGD), achieves optimal in-expectation sample complexity, improving upon existing results for Clip-SGD. Moreover, our result matches the sample complexity of Normalized SGD proved in \citep{hubler2024from}; however, in contrast to the bounds established therein, our guarantees also hold for the stronger convergence measure given by the sum of \textit{squared} gradients.

Further, we show that PS-Clip-SGD also converges with high probability. In this setting, the sample complexity matches that of the in-expectation result up to a multiplicative $\mathrm{polylog}(T/\delta)$ factor.

We evaluate the performance of our method through numerical experiments and compare it to standard baselines. In particular, we demonstrate that, even with untuned hyperparameters, PS-Clip-SGD significantly outperforms both Clip-SGD and Normalized SGD on the task of minimizing a quadratic function with Pareto noise. Additionally, we evaluate PS-Clip-SGD with momentum on an image classification task using AlexNet and CIFAR-100. We show that, in this setting as well, our method outperforms both vanilla SGD with momentum and Clip-SGD with momentum, even when accounting for the additional computational cost induced by per-sample clipping.

Finally, we show that, when gradient accumulation is performed, clipping gradients at each mini-batch can outperform the standard practice of applying clipping only once after all accumulation steps have been completed. We demonstrate this result by training a GPT-2 model with approximately $124$M parameters on the OpenWebText dataset.

\subsection{Related work}

In the presence of light-tailed noise, or when the variance is bounded, the convergence properties of SGD have been extensively studied. Under bounded variance assumptions, convergence in expectation has been established, for example, in \citep{NIPS2011_40008b9a, pmlr-v97-mansel, pmlr-v108-gorbunov20a} for convex functions, and in \citep{ghadimi_lan_2013_stochastic_nonconvex} for the non-convex setting. Under slightly weaker conditions, \citep{khaled2022better} also prove convergence in expectation for non-convex functions. When the noise is assumed to be sub-Gaussian, high-probability convergence results have been obtained, for instance, in \citep{li_orabona_2020_adaptive_sgd_momentum}, while \citep{li_liu_2022_heavy_tails_sgd, madden2024subweibull} further relax this assumption to sub-Weibull noise. Note, however, that these assumptions are still significantly stronger than merely requiring bounded $p$-th moments.

When the gradient noise is heavy-tailed, i.e. when it has a finite $p$-th moment for some $p \in (1,2)$ but not necessarily finite variance, additional regularization techniques seem to become necessary, both in theory and practice. One of the most widely used approaches for improving robustness under heavy-tailed noise is gradient clipping.

Clip-SGD is a commonly used method to stabilize training and mitigate the exploding gradient problem in deep neural networks. Its convergence properties, both in expectation and with high probability, have been studied in various settings. \citep{zhang2020adaptive} establish convergence in expectation for Clip-SGD in non-convex problems with heavy-tailed gradient noise. \citep{cutkosky_mehta_2021_heavy_tails} prove high-probability convergence for a variant that combines clipping with normalization and momentum, achieving rates that match those of \citep{zhang2020adaptive} up to a $\log(T/\delta)$ factor. Subsequently, \citep{nguyen_ene_nguyen_2023_clipped_sgd} establish high-probability convergence of Clip-SGD under slightly weaker assumptions, without requiring momentum or normalization, for both convex and non-convex problems in the heavy-tailed regime. These results were later improved by \citep{improved-convergence}, who removed the additional $\log(T)$ factor in the high-probability upper bound. However, their approach requires a specific choice of $\log(1/\delta)$ and $\Delta_1$-dependent step sizes and clipping thresholds to achieve these convergence rates. Moreover, their guarantees remain suboptimal compared to the lower bound in \citep{zhang2020adaptive}.

\section{Preliminaries}

\paragraph{Notation.}
We use $\mathbb{N}$ and $\mathbb{R}$ to denote the sets of natural and real numbers, respectively. The objective function is denoted by $F:\mathbb{R}^d \to \mathbb{R}$, as in~(\ref{eq: objective function}), where $d \in \mathbb{N}_{\geq 1}$ is the dimension of the parameter space. The constant $L > 0$ denotes the smoothness parameter of $f$, and $\nabla f(\cdot,\cdot)$ represents the stochastic gradient oracle.

Throughout the paper, $T \in \mathbb{N}_{\geq 1}$ denotes the total number of optimization steps. The step size and batch size are denoted by $\eta_t > 0$ and $n \in \mathbb{N}_{\geq 1}$, respectively. With this notation, the vanilla stochastic gradient descent (SGD) update would be given by
\begin{equation*}
    x_{t+1} = x_t - \eta_t \frac{1}{n} \sum_{i=1}^n \nabla f(x_t, \xi_t^{(i)}), \quad t = 1,\ldots,T,
\end{equation*}
with i.i.d samples $\xi_t^{(1)}, \ldots, \xi_t^{(n)} \sim \mathcal{D}$.

For a vector $v \in \mathbb{R}^d$, we denote its Euclidean norm by $|v|$. Since $F$ is not assumed to be convex, the goal is, for a given $\epsilon > 0$, to find a point $x \in \mathbb{R}^d$ such that $|\nabla F(x)| \leq \epsilon$.

Throughout the paper, we make the following standard assumptions.

\begin{assumptions}\label{ass:assumptions for standard SGD}
We assume access to an unbiased stochastic gradient oracle $\nabla f(x,\xi)$ with finite $p$-th moment for some $p \in (1,2]$. That is, there exists $\sigma \in (0,\infty)$ such that for all $x \in \mathbb{R}^d$:
\begin{itemize}
    \item[(i)] $\mathbb{E}\big[\nabla f(x,\xi)\big] = \nabla F(x)$,
    \item[(ii)] $\mathbb{E}\big[|\nabla f(x,\xi)|^{p}\big] \le \sigma^p$.
\end{itemize}

Further, the function $F$ is assumed to be $L$-smooth, i.e., $F$ is differentiable and for all $x_1, x_2 \in \mathbb{R}^d$,
\[
    |\nabla F(x_1) - \nabla F(x_2)| \leq L |x_1 - x_2|.
\]

Finally, we assume that $F$ is bounded from below by some $F^* \in \mathbb{R}$, and we denote the initialization gap by $\Delta_1 := F(x_1) - F^*$.
\end{assumptions}

\section{Main results}
In this section, we present our convergence results for per-sample clipped stochastic gradient descent under Assumptions~\ref{ass:assumptions for standard SGD}. For a fixed batch size $n \in \mathbb{N}$, and at each iteration $t = 1, \ldots, T$, let $\xi_t^{(1)}, \ldots, \xi_t^{(n)} \sim \mathcal{D}$ be i.i.d. with  corresponding stochastic gradients $\nabla f(x_t, \xi_t^{(1)}), \ldots, \nabla f(x_t, \xi_t^{(n)})$.

For fixed parameters $\alpha,\beta > 0$, we define the gradient estimator
\begin{equation}\label{eq: definition of estimator}
    G(x_t,\alpha,\beta) = \frac{1}{n} \sum_{k=1}^n \gamma_t^{(k)} \nabla f(x_t, \xi_t^{(k)}),
\end{equation}
where the clipping factors $\gamma_t^{(k)}$ are defined as
$
    \gamma_t^{(k)} \coloneqq \min \left\{1, \frac{\alpha k^{\frac{1}{\beta}}}{\left|\nabla f(x_t, \xi_t^{(k)})\right|} \right\}.
$

The parameter updates are then given by $x_{t+1} = x_t - \eta_t G(x_t,\alpha,\beta)$. We refer to this algorithm as per-sample clipped SGD (PS-Clip-SGD).

In the next section, we begin by establishing convergence of PS-Clip-SGD in expectation. Then, in Section~\ref{sec: high prob convergence}, we show that these convergence guarantees also hold with high probability. Finally, in Section~\ref{sec: Experiments}, we evaluate our method on a simple quadratic problem with additive noise, as well as on image recognition and language modeling tasks.

\subsection{Convergence in expectation}
The following theorem characterizes the convergence rate of PS-Clip-SGD in expectation.

\begin{theorem}\label{thm: in expectation convergence of SGD}
Under Assumptions \ref{ass:assumptions for standard SGD} and for any $\eta_t < 1/L$, the iterates generated by PS-Clip-SGD with $\alpha = \sigma$ and $\beta = p$ satisfy
\begin{align*}
\sum_{t=1}^T\frac{\eta_t\E(|\nabla F(x_t)|^2)}{{\sum_{t=1}^T\eta_t}} \leq \frac{2\Delta_1}{\sum_{t=1}^T\eta_t} +  8\sigma^2 n^{-\frac{2(p-1)}{p}}\;.
\end{align*}
\end{theorem}
Note that we do not require a specific choice of $\eta_t$ to obtain this convergence rate, as long as the step size satisfies $\eta_t < 1/L$. 

The idea for the proof of this theorem is rather simple. Using the robustness of the estimator $G(x,\alpha,\beta)$, it is possible to show that
$
    \mathbb{E}\bigl[|\nabla F(x) - G(x,\sigma,p)|^2\bigr] = O\!\left(n^{-\frac{2(p-1)}{p}}\right)
$
for any $x \in \mathbb{R}^d$. Note that this would not be possible using a simple empirical mean, since the gradient noise is not assumed to have bounded variance. Once this error estimate is established, the remainder of the proof follows in a fairly standard way by using a Taylor expansion of $F$ and a telescoping sum to bound the norm of the gradients. The detailed proof of Theorem \ref{thm: in expectation convergence of SGD} can be found in Appendix \ref{Proof of in-expectation convergence}.

By an appropriate choice of step size and batch size, the convergence rate in Theorem \ref{thm: in expectation convergence of SGD} can be rewritten as follows.

\begin{corollary}\label{cor: in expectation upper bound}
Let $\eta_t = 1/(2L)$, and $n = \left\lceil \max \left\{1, \left(\frac{\sigma^2 T}{\Delta_1 L}\right)^{\frac{p}{2(p-1)}}\right\} \right\rceil$. Then, under Assumptions \ref{ass:assumptions for standard SGD}, we have that the iterates generated by PS-Clip-SGD with $\alpha = \sigma$ and $\beta = p$ satisfy
\begin{align*}
\frac{1}{T}\sum_{t=1}^T\E(|\nabla F(x_t)|^2)\leq 9\frac{\Delta_1 L}{T}\;.
\end{align*}
This corresponds to a sample complexity of
$
{O}\left(
\frac{\Delta_1 L}{\varepsilon^2}
+
\frac{\Delta_1 L}{\varepsilon^2}
\left(\frac{\sigma}{\varepsilon}\right)^{\frac{p}{p-1}}
\right)$ to achieve $\frac{1}{T}\sum_{t=1}^T\E(|\nabla F(x_t)|^2) <\epsilon^2$.
\end{corollary}
Note, that the sample complexity matches the lower bound proved in \citep{zhang2020adaptive}.
To the best of our knowledge, this result improves upon the best known in-expectation convergence rates under heavy-tailed noise of both Clip-SGD and Normalized SGD. By comparison, using the same batch size as in the previous corollary and the step size $\eta_t = \sqrt{\Delta_1/LT}$, \citep{hubler2024from} prove the upper bound
\begin{equation}\label{eq: exp upper bound normalized SGD}
\frac{1}{T}\sum_{t=1}^T\E(|\nabla F(x_t)|) \leq 6\sqrt{\frac{\Delta_1 L}{T}}
\end{equation}
for Normalized SGD. Using Jensen's inequality, it is easy to see that the bound in Corollary \ref{cor: in expectation upper bound} is better than \eqref{eq: exp upper bound normalized SGD} by a factor of 2. Additionally, the bound proved by \citep{hubler2024from} only holds for the average of gradient norms, and the authors show that using normalized SGD, it is not possible to obtain a similar upper bound for the average of \textit{squared} gradient norms, which is a strictly stronger measure. Finally, note that the optimal step size used to establish the bound in \eqref{eq: exp upper bound normalized SGD} requires knowledge of the initialization gap $\Delta_1$.

One clear drawback of our method compared to Normalized SGD is that it requires knowledge of the problem parameters to define the gradient estimator. Specifically, it requires an upper bound on the maximal $p$-moment $\sigma^p$ of $|\nabla f(x,\xi)|$. However, we will see in the numerical experiments (Section \ref{sec: Experiments}), that our method is able to performs well even with untuned and unrefined choices of hyperparameters.

\subsection{Convergence with high probability}\label{sec: high prob convergence}

While in-expectation results guarantee convergence given sufficiently many runs of the optimization algorithm, in practice it is often infeasible to run the algorithm enough times for this type of convergence to be relevant. For this reason, it is typically more desirable to establish \emph{high-probability} convergence guarantees. Specifically, the goal is to show that, with probability at least $1-\delta$, the relevant convergence measure, such as the average of squared gradient norms, is bounded by a term with a polylogarithmic dependency on $\log(1/\delta)$.

The next theorem provides a high probability convergence result for PS-Clip-SGD.
\begin{theorem}\label{thm: high prob covergence of SGD}
    Let $\delta \in (0,1)$. Then, under Assumptions \ref{ass:assumptions for standard SGD} and for any $\eta_t <1/L$, we have that the iterates generated by PS-Clip-SGD with $\alpha = \frac{\sigma}{(\log(1/\delta)+1/4)^{1/p}}$ and $\beta = p$ satisfy
\begin{align*}
        \sum_{t=1}^T\frac{\eta_t|\nabla F(x_t)|^2}{{\sum_{t=1}^T\eta_t}} \leq \frac{2\Delta_1}{\sum_{t=1}^T\eta_t} +49\sigma^2\left(\frac{\log(1/\delta)+\log(T)+1/4}{n}\right)^{\frac{2(p-1)}{p}}
    \end{align*}
    with probability at least $1-\delta$.
\end{theorem}

The proof idea for this result is analogous to the one for its in-expectation counterpart. We first establish a high-probability upper bound for $|\nabla f(x)-G(x,\alpha,\beta)|^2$. Specifically, using a similar argument to \citep{bandits}, we show that with probability at least $1-\delta$
\begin{equation*}
    |\nabla F(x)-G(x,{\sigma}/{(\log(1/\delta)+1/4)^{1/p}},p)|^2\leq  49\sigma^2\left(\frac{\log(1/\delta)+1/4}{n}\right)^{\frac{2(p-1)}{p}} \;.
\end{equation*}
Union bound is then applied to estimate the maximum of this error over the $T$ steps. The rest of the proof follows the same steps as the in-expectation convergence result.
The detailed poof of Theorem \ref{thm: high prob covergence of SGD} can be found in Appendix \ref{Proof of high-probability convergence}. 

Similarly to its in-expectation counterpart, using an appropriate choice of step and batch size, Theorem \ref{thm: high prob covergence of SGD} implies the following corollary.
\begin{corollary}\label{cor: high prob convergence}
    Let $\delta \in (0,1)$, $\eta_t = 1/(2L)$ and $n =\left\lceil\max \left\{1, \left(\frac{\sigma^2 T}{\Delta_1 L}\right)^{\frac{p}{2(p-1)}}\right\}\right\rceil$. Then, under Assumptions \ref{ass:assumptions for standard SGD}, we have that the iterates generated by PS-Clip-SGD with $\alpha = \frac{\sigma}{(\log(1/\delta)+1/4)^{1/p}}$ and $\beta = p$ satisfy
    \begin{align*}
        \frac{1}{T}\sum_{t=1}^T|\nabla F(x_t)|^2\leq \left(3+49\left(\log(\delta^{-1})+\log(T)\right)\right)^{\frac{2(p-1)}{p}}\frac{\Delta_1 L}{T}\;
    \end{align*}
    with probability at least $1-\delta$.
\end{corollary}
Compared to the high-probability convergence rate proved in~\citep{hubler2024from}, Corollary~\ref{cor: high prob convergence} incurs an additional factor of $\log(T)^{\frac{p-1}{p}}$. On the other hand, our convergence rate has a considerably improved dependence on the failure probability $\delta$. As we demonstrate in Section~\ref{sec: Experiments}, this improved $\log(1/\delta)$ dependence is also observed empirically. Additionally, as in the in-expectation setting, the optimal step size used in \citep{hubler2024from} requires knowledge of the initialization gap $\Delta_1$, and the bound proved therein does not hold for the sum of \textit{squared} gradient norms.

To the best of our knowledge, the strongest high-probability convergence rate for Clip-SGD is given in~\citep{improved-convergence}, where the authors establish an upper bound of order
\begin{equation}\label{eq: improved eq} O\left({\log(1/\delta)^{\frac{p}{p-1}}}{T^{\frac{2(1-p)}{3p-2}}}\right)
\end{equation}
for the sum of squared gradient norms. We note that our result achieves both a better dependence on $\log(1/\delta)$ than \eqref{eq: improved eq} and an improved sample complexity.

Furthermore, the result in \citep{improved-convergence} requires a specific $\log(1/\delta)$- and $\Delta_1$-dependent choice of step sizes and clipping thresholds to obtain the stated convergence rate. In contrast, in Theorem~\ref{thm: high prob covergence of SGD}, we only require $\eta_t < 1/L$. Moreover, the following remark shows that the clipping threshold in our result can also be chosen independently of $\delta$ while still ensuring high-probability convergence. This observation suggests that our algorithm should be less sensitive to the choice of hyperparameters than existing methods.

\begin{remark}
If in Theorem~\ref{thm: high prob covergence of SGD} we assume $\delta \leq e^{-3/4}$, and choose $\alpha = \sigma$ as in Theorem~\ref{thm: in expectation convergence of SGD}, then with probability at least $1 - \delta$, we have that
\begin{align*}
        \sum_{t=1}^T\frac{\eta_t|\nabla F(x_t)|^2}{{\sum_{t=1}^T\eta_t}} \leq \frac{2\Delta_1}{\sum_{t=1}^T\eta_t} +49\sigma^2\frac{\left(\log(1/\delta)+\log(T)+1/4\right)^2}{n^{\frac{2(p-1)}{p}}}\;.
    \end{align*}
    Hence, if we choose $n =\left\lceil\max \left\{1, \left(\frac{\sigma^2 T}{\Delta_1 L}\right)^{\frac{p}{2(p-1)}}\right\}\right\rceil$, we obtain that with probability $1-\delta$,
     \begin{equation}\label{eq: delta ind bound}
        \frac{1}{T}\sum_{t=1}^T|\nabla F(x_t)|^2\leq \left(3+49\left(\log(\delta^{-1})+\log(T)\right)\right)^2\frac{\Delta_1 L}{T}\;.
    \end{equation}
\end{remark}

Note that even in (\ref{eq: delta ind bound}) the dependency on $\log(1/\delta)$ is better than in (\ref{eq: improved eq}), since $p/(p-1)\geq2$ for any $p\in(1,2]$.

% \begin{equation}
%     \frac{\left(\Delta_1 L\right)^{\frac{3p-2}{2p-1}}}{\epsilon^{\frac{6p-4}{2p-1}}}+\frac{\left(\sigma^{2p}(\Delta_1 L)^{p-2}\right)^{\frac{3p-2}{2(p-1)}}}{}
% \end{equation}

\section{Experiments}\label{sec: Experiments}
\subsection{Quadratic function with noise}
We compare the performance of Normalized-SGD, Clip-SGD and PS-Clip-SGD on a simple quadratic problem. Specifically, we minimize the function $f(x,\xi)=\frac{1}{2}|x|^2+\langle x,\xi\rangle$, where $x\in \R^{10}$ and $\xi\in \R^{10}$ is a random vector with i.i.d. components drawn from a symmetrized Pareto distribution with tail index $p >1$. To estimate the gradient at a point $x$, we use a batch size of $n=64$, where each sample $y_i$ in the batch has the form $y_i = x +\xi_i$ and new $\xi_i \in \R^{10}$ are sampled at each iteration.

We recall that if $g(x)$ is an unbiased gradient estimator, the parameter updates for Normalized-SGD and Clip-SGD with step size $\eta>0$ and clipping threshold $\gamma$ are given by $x_{t+1}= x_t-\eta {g(x_t)}/{|g(x_t)|}$ and $
    x_{t+1}= x_t-\eta \min\left\{1, {\gamma}/{|g(x_t)|}\right\} g(x_t)$
    respectively.
% The gradient estimators for the three algorithms are therefore given by
% \begin{equation*}
%     g^T(x,\alpha) = \frac{1}{64}\sum_{k=1}^{64}\min\left\{1, \frac{(\alpha k^{\frac{1}{p}})}{|(x+\xi_i)|}\right\}(x+\xi_i)\;,
% \end{equation*}
% for PS-Clip-SGD, and by
% \begin{equation*}
%     g(x) = \frac{1}{64}\sum_{k=1}^{64}(x+\xi_i)\;,
% \end{equation*}
% for Normalized-SGD and Clip-SGD. The gradient step is then given by
% \begin{align*}
%     x_{t+1}&= x_t-\eta_{PSCSGD} \;g^T(x_t,\alpha), \quad \text{for PS-Clip-SGD}\\
%     x_{t+1}&= x_t-\eta_{NSGD} \; \frac{g(x_t)}{|g(x_t)|}, \quad \text{for Normalized-SGD}\\
%     x_{t+1}&= x_t-\eta_{CSGD} \; \min\left\{1, \frac{\gamma}{|g(x_t)|}\right\} g(x_t), \quad \text{for Clip-SGD .}
% \end{align*}

We run each algorithm for $T = 2000$ iterations with a constant step size $\eta = 0.01$. In this experiment, we do not perform any parameter tuning and instead set $\alpha = \beta = \gamma = 1$. We also conducted the same experiment with tuned hyperparameter for each algorithm; details are provided in Appendix~\ref{app: Quadratic function with noise}. We mention that we did perform the experiment using vanilla SGD as well. However, likely due to the heavy noise, vanilla SGD performed very poorly in our experiments and is therefore not presented here.

Figure~\ref{fig:untuned convergence} shows the average performance over $10$ runs for each algorithm and for different values of the noise parameter $p > 1$. Note that, due to the heavy-tailed noise, the gradient norms are often larger than one. As a result, Normalized SGD and Clip-SGD appear indistinguishable in the figure. We refer to Table~\ref{tab: performance comparison quadratic prob untuned params} for a more detailed comparison of the algorithms performance. From Figure~\ref{fig:untuned convergence}, we can see that PS-Clip-SGD outperforms both Clip-SGD and Normalized SGD across all three noise regimes. Additionally, we observe that while heavier tails significantly slow down the convergence of Clip-SGD and Normalized SGD, the performance of PS-Clip-SGD remains stable across different values of the noise parameter, which confirms the robustness of this method to heavy tailed noise.

\begin{figure}[h]
    \centering
    \includegraphics[width=0.32\textwidth]{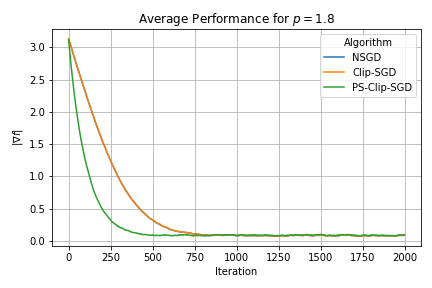}
    \hfill
    \includegraphics[width=0.32\textwidth]{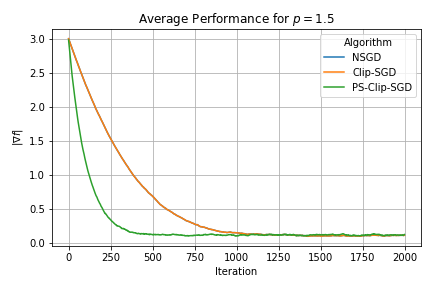}
    \hfill
    \includegraphics[width=0.32\textwidth]{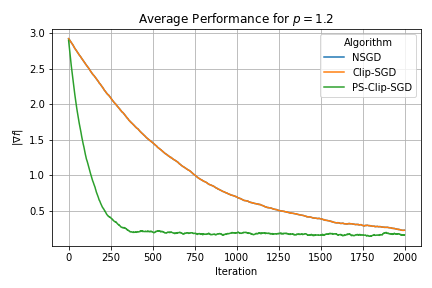}
    \caption{Performance of Normalized-SGD, Clip-SGD and PS-Clip-SGD for different noise regimes without parameter tuning. Due to the choice of parameters, Normalized SGD and Clip-SGD are indistinguishable in the plot, see Table \ref{tab: performance comparison quadratic prob untuned params} for more details.}
    \label{fig:untuned convergence}
\end{figure}
We also wish to empirically investigate the high-probability convergence properties of PS-Clip-SGD. To this end, we run the algorithm $10^4$ times for $T = 100$ iterations and, for each run, compute the average gradient norm over the $T$ iterations. We again perform this experiment without hyperparameter tuning.
Figure~\ref{fig: high prob convergence} shows the $(1 - \delta)$-quantiles of the average gradient norm plotted against $\log(1/\delta)$. As expected, the dependence is sublinear, which is consistent with our theoretical high-probability convergence results. Additionally, we observe that PS-Clip-SGD outperforms both Normalized SGD and Clip-SGD in this benchmark. This is consistent with the improved dependence on $\log(1/\delta)$ in the upper bound for PS-Clip-SGD compared to Normalized SGD and Clip-SGD, as discussed in Section~\ref{sec: high prob convergence}.

\begin{figure}[h]
    \centering
    \includegraphics[width=0.32\textwidth]{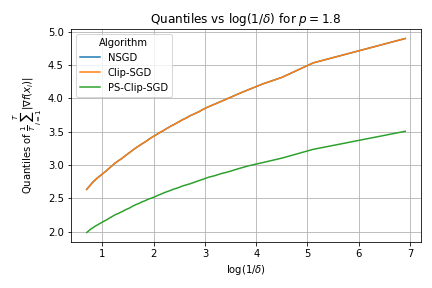}
    \hfill
    \includegraphics[width=0.32\textwidth]{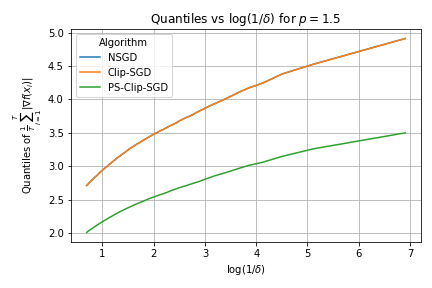}
    \hfill
    \includegraphics[width=0.32\textwidth]{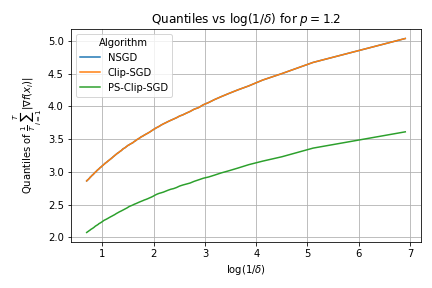}
    \caption{$(1-\delta)$-quantile of the average gradient norm after $T = 100$ training steps, plotted against $\log(1/\delta)$ for the three algorithms and different noise regimes. As before, due to the choice of parameters, Normalized SGD and Clip-SGD are indistinguishable in the plot.}
    \label{fig: high prob convergence}
\end{figure}
\subsection{Training AlexNet with per-sample clipping}\label{sec: Training AlexNet with per-sample clipping}
\begin{figure}[h]
    \centering
    \includegraphics[width=0.35\textwidth]{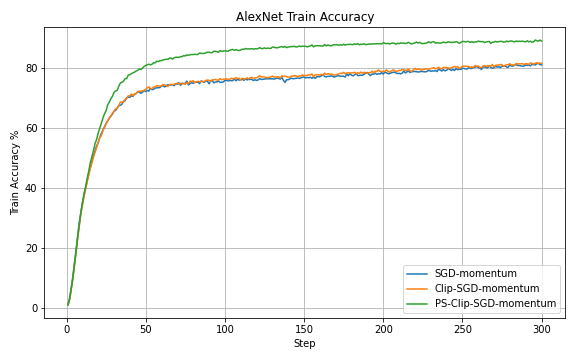}
    \includegraphics[width=0.35\textwidth]{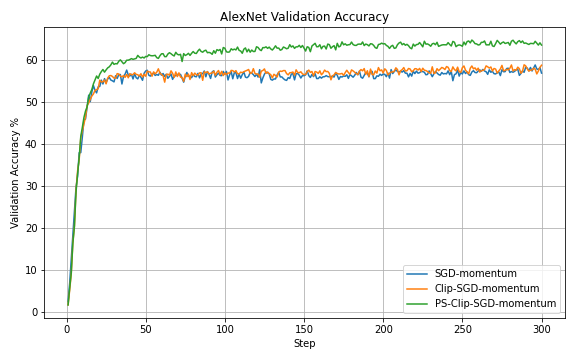}
    \hfill
    \caption{Training and validation accuracies of SGD, Clip-SGD and PS-Clip-SGD, all with momentum, for training AlexNet on the CIFAR-100 dataset}
    \label{fig: AlexNet performance}
\end{figure}
\begin{table}[h]
\caption{Performance comparison of SGD, Clip-SGD and PS-Clip-SGD, all with momentum, for training AlexNet on the CIFAR-100 dataset.}
\centering
\begin{tabular}{lccccc}
    \toprule
    Algorithm & \makecell{Train\\Accuracy}&  \makecell{Val\\Accuracy} &\makecell{Test\\Accuracy}  & \makecell{Total Training\\Time (min)} & \makecell{Val Accuracy \\ after 119 min} \\
    \midrule
    SGD-mom & 81.04 & 58.90 & 57.96 &\textbf{119} &58.90 \\
    Clip-SGD-mom & 81.16 & 58.94&58.84 & 122 & 58.94 \\
    PS-Clip-SGD-mom & \textbf{88.85} & \textbf{64.84} & \textbf{64.64}& 162 &  \textbf{64.40}\\
    \bottomrule
\end{tabular}\label{tab: AlexNet training}
\end{table}
In this experiment, we train AlexNet \citep{alexnet-paper} on the CIFAR-100 dataset \citep{CIFAR-100} and compare the performance of vanilla SGD, Clip-SGD, and PS-Clip-SGD. In all three cases, we use momentum and weight decay. The model is trained for 300 epochs with a learning rate of 0.01, momentum parameter $0.9$ and weight decay of $10^{-4}$. The clipping threshold is set to $45.0$ for PS-Clip-SGD and $15.0$ for Clip-SGD. We refer to Appendix \ref{app: alexnet} for more details on the choice of hyperparameters. Figure~\ref{fig: AlexNet performance} shows the training and validation accuracy for the three methods, while Table~\ref{tab: AlexNet training} provides additional details on training performance.

From the last column of Table~\ref{tab: AlexNet training}, we observe that, even after accounting for the additional computational cost, PS-Clip-SGD clearly outperforms both vanilla SGD and Clip-SGD. Moreover, Figure~\ref{fig: AlexNet performance} shows that PS-Clip-SGD reaches a validation accuracy above $60\%$ after only 40 epochs (approx. 22 minutes), whereas the other methods fail to reach $59\%$ throughout training.

The similar performance of standard SGD and Clip-SGD can be explained by the high clipping threshold: only about $0.13\%$ of gradients are clipped during training, so Clip-SGD behaves mostly like its unclipped counterpart. Different clipping thresholds and step-sizes were also tested for Clip-SGD, but did not improve performance; see Appendix \ref{app: alexnet} for details.

In contrast, per-sample clipping leads to a clear performance improvement. Our results suggest that this may be due to the presence of a small number of outliers with very large gradients, which are effectively controlled by PS-Clip-SGD, thereby improving convergence. Additional details are provided in Appendix~\ref{app: alexnet}.

Despite the clear advantages of PS-Clip-SGD, it is important to discuss the additional computational cost associated with per-sample gradient computations, which we consider the main drawback of this method. For this experiment, we used Opacus' Ghost Clipping API \citep{lee_kifer_2021_fast_clipping_popets, ghost_clipping, opacus_clipping_blog_2024} to compute per-sample gradient norms. Although this enables efficient per-sample computations, PS-Clip-SGD was still approximately $1.3\times$ slower than standard Clip-SGD. In this case, the improved convergence speed was sufficient to offset the additional computational time. However, this may not hold for larger architectures or batch sizes. For instance, when evaluating our method on a GPT-2 model, the performance gains from per-sample clipping were outweighed by the increased computational time, making the method impractical for that task.

In the next section, however, we demonstrate that, when using gradient accumulation, mini-batch gradient clipping can achieve performance improvements similar to per-sample clipping while incurring virtually no additional computational cost.

\subsection{Training GPT-2 with mini-batch clipping}\label{sec: Training GPT-2 with mini-batch clipping}

A common method to overcome hardware limitations when training large models is gradient accumulation. Given a mini-batch size $m \in \mathbb{N}$ and a number of accumulation steps $k \in \mathbb{N}$, gradient accumulation allows one to \textit{simulate} a batch size of $m \cdot k$ by summing the gradients computed at each accumulation step and performing a parameter update only after $k$ iterations.

To the best of our knowledge, when gradient accumulation is combined with gradient clipping, the general consensus is that gradients should be clipped only once after all $k$ accumulation steps have been completed, with many sources explicitly advising against clipping each mini-batch gradient \citep{aiwiki_gradient_accumulation,apxml_gradient_clipping_accumulation, brenndoerfer_2026_gradient_accumulation,karpathy_nanogpt_2022}. The rationale is that gradient accumulation is intended to approximate the gradient that would be obtained using a single batch of size $m \cdot k$, and clipping at each accumulation step would introduce an unwanted bias and lead to wrong gradients. However, based on the insights from the previous sections, we argue that clipping the gradients computed after each accumulation step can actually improve the training performance of large models. Note that, in contrast to per-sample gradient clipping, this approach introduces virtually no additional computational cost, since the gradients for each mini-batch are computed regardless of when clipping is applied.

% We specifiy that by clipping at each accumulation step we mean that for a mini-batch size of $m$ and number of accumulation steps $k$ the gradient estimator at step $t$ is given by

% \begin{equation*}
%     \sum_{i = 1}^k \left(\sum_{j=1}^m \nabla f(x_t,\xi_t^{(j,k)})\right)\min\left\{1, \frac{\gamma_t}{\left|\sum_{j=1}^m \nabla f(x_t,\xi_t^{(j,k)})\right|} \right\} \;.
% \end{equation*}

In this experiment, we trained a GPT-2 model with $124$M parameters \citep{karpathy_nanogpt_2022} from scratch using the OpenWebText dataset \citep{Gokaslan2019OpenWeb}. We used the same training loop for both mini-batch-clipped SGD (MB-Clip-SGD) and \textit{standard} Clip-SGD. The only difference being that, for MB-Clip-SGD, we clip the gradients computed for each mini-batch, whereas for standard Clip-SGD, clipping is applied only once after all accumulation steps have been completed.

 The model was trained for $5000$ steps using the AdamW optimizer. We increase the learning rate linearly for the first 500 steps and use cosine learning rate decay after that; further details are provided in Appendix~\ref{app:Training GPT-2}. Gradient clipping was applied with a threshold of $1.0$ for both methods. We used gradient accumulation with $64$ accumulation steps and a mini-batch size of $8$, which is the largest mini-batch size that fits on our hardware. This corresponds to approximately $5\cdot 10^5$ tokens per batch. Note that this batch size and clipping threshold are the same used to train the $124$M parameters model in \citep{GPT3-paper}. We perform the experiment with three learning rates, the results are summarized in Table \ref{tab: performance comparison GPT-Training}. As we can see, mini-batch clipping outperforms standard gradient clipping regardless of the learning rate. While the performance improvement is relatively modest, we believe this to be an interesting observation, as it contradicts the prevailing view that clipping should be applied only after all accumulation steps have been completed.

\begin{table}[h]
  \caption{Performance comparison of Clip-SGD and MB-Clip-SGD for training GPT-2 with different learning rates.}
  \label{tab: performance comparison GPT-Training}
  \centering
  \begin{tabular}{clcc}
    \toprule
    Learning Rate    & Algorithm  & Final Val Loss & \\
    \midrule
    \multirow{2}{*}{$6\cdot 10^{-4}$} 
        & Clip-SGD         & $3.3495$ \\
        & MB-Clip-SGD      & $\mathbf{3.3405}$ \\
    \midrule
    \multirow{2}{*}{$8\cdot 10^{-4}$} 
        & Clip-SGD         & $3.3114$  \\
        & MB-Clip-SGD      & $\mathbf{3.3059}$ \\
    \midrule
    \multirow{2}{*}{$1\cdot 10^{-3}$} 
        & Clip-SGD         & $3.2922$  \\
        & MB-Clip-SGD      & $\mathbf{3.2830}$  \\
    \bottomrule
  \end{tabular}
\end{table}

% Figure~\ref{fig:GPT-2 training} shows the validation loss for the two clipping methods. Since the loss starts at values around $10.82$, we plot only the steps from $500$ onward to make the differences more visible.

% \begin{figure}[h]
% \centering
% \includegraphics[width=0.5\textwidth]{GPT_val_loss.png}
% \caption{Performance of mini-batch gradient clipping (MB-Clipping) and standard gradient clipping for training a GPT-2 model with $124$M parameters.}
% \label{fig:GPT-2 training}
% \end{figure}
\section{Conclusion}
In this work, we propose per-sample gradient clipping as an alternative to standard clipping for improving training performance in the presence of heavy-tailed gradient noise. We show that this method provides strong convergence guarantees and support our theoretical findings with various numerical experiments. Additionally, we argue that when gradient accumulation is used, mini-batch clipping can serve as a more efficient alternative to per-sample clipping to improve convergence speed.

Although our theoretical and empirical results are promising, we would like to discuss some limitations of our analysis and outline possible directions for future work. First, we note that the bounded $p$-moment assumption on the gradient noise considered in this work is, in general, stronger than assuming a bounded \textit{centered} $p$-moment. In our setting, this assumption is required to control the bias introduced by per-sample clipping. It remains an open question whether similar results can be established under the weaker assumption of a bounded centered $p$-moment.

The definition of the clipping threshold used in this work requires knowledge of at least a lower bound on the noise parameter $p$ and an upper bound on the corresponding moment $\sigma^p$. While convergence of PS-Clip-SGD can still be established for arbitrary $\alpha > 0$ and $\beta \in (1,2]$, the resulting rates are not optimal. It remains unclear whether it is possible to derive optimal, parameter-free convergence rates, for example matching those obtained for Normalized SGD in \citep{hubler2024from}.

Our theoretical analysis requires a clipping threshold that increases with each sample. On the other hand, our experiments on real data suggest that a constant clipping threshold may yield better performance in practice; see also Appendix~\ref{app: Increasing vs constant clipping threshold}. This observation is consistent with the theory and practice of Clip-SGD, where increasing or very large clipping thresholds are often required in theory, whereas in practice small constant thresholds are commonly used. An interesting open question is whether in our setting optimal convergence rates can be established when using a constant, even if large, clipping threshold.

\newpage
\bibliography{literature}
\newpage
\appendix
\section{Known results used in the proofs}
\begin{lemma}[Vector Bernstein Inequality, see {\citep[Lemma~18]{kohler2017subsampled}}]\label{lem: vector berstein inequality}
Let $X_1, \ldots, X_n$ be independent vector-valued random variables with common dimension $d$, satisfying
\begin{equation*}
    \mathbb{E}[X_i] = 0 \;,
\quad
|X_i| \le c 
\quad \text{and} \quad 
\mathbb{E}\big[|X_i|^2\big] \le \sigma^2,
\end{equation*}
for all $i= 1,\ldots,n$ and some $c,\sigma>0$. Then, for every $\epsilon>0$ we have that
\begin{equation*}
    \mathbb{P}\left(\left|\frac{1}{n}\sum_{i=1}^nX_i\right| \ge \epsilon\right) 
\leq
\exp\left(\frac{-n\epsilon^2}{8\sigma^2+c\epsilon} + \frac{1}{4}\right).
\end{equation*}
In particular, this implies that for any $\delta \in (0,1)$
\begin{equation*}
    \mathbb{P}\left(\left|\frac{1}{n}\sum_{i=1}^nX_i\right| \geq \frac{c \log(1/\delta)}{n}+\sqrt{\frac{8\sigma^2\log(1/\delta)}{n}}\right)\leq \delta e^{1/4} \;.
\end{equation*}
\end{lemma} 
\section{Proof of in-expectation convergence}\label{Proof of in-expectation convergence}
The following Lemma provides a generic in-expectation error bound for the clipped mean estimator.
\begin{lemma}\label{lemm: generic convergence in espectation}
    Let $Y_1,\ldots, Y_n \sim Y$ be i.i.d. random variables such that $\E\left|Y\right|^{p}\leq \sigma^p<\infty$ for some $p\in (1,2]$. For any $\alpha\geq\sigma$ define 
    \begin{equation*}
        \gamma_k = \min\left\{1,\frac{\alpha k^{1/p}}{|Y_k|}\right\}\;.
    \end{equation*}
    Then,
    \begin{equation*}
    \E\left|\frac{1}{n}\sum_{k= 1}^nY_k\gamma_k-\E(Y)\right|^2 \leq 8\alpha^2n^{-\frac{2(p-1)}{p}}\;.
    \end{equation*}
\end{lemma}
\begin{proof}
    We set, $\mu = \E(Y)$ and $\mu_k = \E(Y_k\gamma_k)$. Then,
    \begin{equation}\label{eq: expectation bias}
    \begin{aligned}
        |\mu_k-\mu| &= \left|\E\left(\left(\frac{\alpha k^{\frac{1}{p}}}{|Y_k|}-1\right)Y_k\1_{(|Y_k|>\alpha k^{1/p})}\right)\right|\leq \E\left(Y_k1_{(|Y_k|>\alpha k^{1/p})}\right)\\
        &\leq \E\left(|Y_k|^{p}\1_{(|Y_k|>\alpha k^{1/p})}\right)\alpha^{1-p} k^{\frac{1-p}{p}}\leq \sigma^p\alpha^{1-p} k^{\frac{1-p}{p}} \;.
    \end{aligned}
    \end{equation}
    Hence, we obtain that
    \begin{equation}\label{eq: bias term}
    \begin{aligned}
        \E\left|\frac{1}{n}\sum_{k= 1}^n\left(\mu_k-\mu\right)\right|^2 &\leq \frac{1}{n^2}\left(\sum_{k= 1}^n\sigma^p\alpha^{1-p} k^{\frac{1-p}{p}}\right)^2
         \leq \frac{\sigma^{2p}\alpha^{2(1-p)}}{n^2}\left(2n^{\frac{1}{p}}\right)^2\\
         &= 4\sigma^{2p}\alpha^{2(1-p)}n^{-\frac{2(p-1)}{p}}\leq 4 \alpha^2n^{-\frac{2(p-1)}{p}}\;.
    \end{aligned}
    \end{equation}
    Further, we have that
    \begin{equation}\label{eq: lln term}
    \begin{aligned}
        \E\left|\frac{1}{n}\sum_{k= 1}^n\left(Y_k\gamma_k-\mu_k\right)\right|^2 &=\frac{1}{n^2}\sum_{k= 1}^n\E\left|Y_k\gamma_k-\mu_k\right|^2 \leq \frac{1}{n^2}\sum_{k= 1}^n\E\left|Y_k\gamma_k\right|^2\\
        &\leq \frac{1}{n^2}\sum_{k= 1}^n\E\left|Y_k\right|^{p}\alpha^{2-p} k^\frac{2-p}{p} \leq \frac{4\sigma^p}{n^2}\alpha^{{2-p}}n^\frac{2}{p}\\
        &=4\sigma^{p}\alpha^{2-p}n^{-\frac{2(p-1)}{p}}\leq 4 \alpha^2n^{-\frac{2(p-1)}{p}}\;,
    \end{aligned}
    \end{equation}
    where in the second inequality we used the fact that $\left|Y_k\gamma_k\right|^2 = \left|Y_k\gamma_k\right|^p\left|Y_k\gamma_k\right|^{2-p}\leq \left|Y_k\gamma_k\right|^p\alpha^{2-p} k^\frac{2-p}{p}$ by definition of $\gamma_k$.
    Combining (\ref{eq: bias term}) and (\ref{eq: lln term}), we obtain that
    \begin{align*}
        \E\left|\frac{1}{n}\sum_{k= 1}^nY_k\gamma_k-\mu\right|^2 &\leq \E\left|\frac{1}{n}\sum_{k= 1}^n\left(Y_k\gamma_k-\mu_k\right)\right|^2 + \E\left|\frac{1}{n}\sum_{k= 1}^n\left(\mu_k-\mu\right)\right|^2\leq 8\alpha^2n^{-\frac{2(p-1)}{p}} \;,
    \end{align*}
    which concludes the proof.
    \end{proof}
Setting $\alpha = \sigma$ in the previous Lemma we obtain the following corollary.
\begin{corollary}\label{cor: generic convergence in espectation with known parameters}
    Let $Y_1,\ldots, Y_n \sim Y$ be i.i.d. random variables such that $\E\left|Y\right|^{p}\leq \sigma^p<\infty$ for some $p\in (1,2]$. Further, define 
    \begin{equation*}
        \gamma_k = \min\left\{1,\frac{\sigma k^{1/p}}{|Y_k|}\right\}\;.
    \end{equation*}Then,
    \[\E\left|\frac{1}{n}\sum_{k= 1}^nY_k\gamma_k-\E(Y)\right|^2 \leq 8\sigma^2 n^{-\frac{2(p-1)}{p}}\;.\]
\end{corollary}
\begin{proposition}\label{prop: standard bound on iterates}
    Under Assumptions \ref{ass:assumptions for standard SGD} and for any $\eta_t<1/L$, the iterates generated by PS-Clip-SGD satisfy
    \begin{align*}
            \sum_{m=1}^T\eta_t|\nabla F(x_t)|^2 &\leq 2\Delta_1+\sum_{m=1}^T\eta_t|\nabla F(x_t)-G(x_t,\alpha,\beta)|^2 \;.
        \end{align*}
\end{proposition}
\begin{proof}
    By $L$-smoothness of the function $F$, we have that for every $t = 1, \ldots, T-1$,
    \begin{align*}
        F(x_{t+1})-F(x_t) &\leq -\eta_t\langle\nabla F(x_t),G(x_t,\alpha,\beta)\rangle+\frac{L\eta_t^2}{2}|G(x_t,\alpha,\beta)|^2\\
        &=-\frac{\eta_t}{2}\left(|\nabla F(x_t)|^2+|G(x_t,\alpha,\beta)|^2-|\nabla F(x_t)-G(x_t,\alpha,\beta)|^2\right)+\frac{L\eta_t^2}{2}|G(x_t,\alpha,\beta)|^2\\
        &\leq -\frac{\eta_t}{2}\left(|\nabla F(x_t)|^2-|\nabla F(x_t)-G(x_t,\alpha,\beta)|^2\right) \;,
    \end{align*}
    where the equality in the second line is simply the polarization identity for inner products. Summing over $t$ and telescoping we obtain that
    \begin{align*}
            \sum_{m=1}^T\eta_t|\nabla F(x_t)|^2 &\leq 2\Delta_1+\sum_{m=1}^T\eta_t|\nabla F(x_t)-G(x_t,\alpha,\beta)|^2.
        \end{align*}
\end{proof}
We can now prove Theorem \ref{thm: in expectation convergence of SGD}.
\begin{proof}[Proof of Theorem \ref{thm: in expectation convergence of SGD}]
By proposition \ref{prop: standard bound on iterates} we know that 
\begin{align}\label{eq: telescoped sum}
            \sum_{m=1}^T\eta_t|\nabla F(x_t)|^2 &\leq 2\Delta_1+\sum_{m=1}^T\eta_t|\nabla F(x_t)-G(x_t,\sigma,p)|^2.
        \end{align}
        Further, by Corollary \ref{cor: generic convergence in espectation with known parameters} we have that $\E(|\nabla F(x)-G(x,\sigma,p)|^2)\leq 8\sigma^2 n^{-\frac{2(p-1)}{p}}$ for any fixed $x \in \R^d$. By assumption, for any $t = 1, \ldots,T$ the samples $\xi_t^{(1)},\ldots,\xi_t^{(n)}$ are independent of $x_t$. Hence, we obtain that
        \begin{equation}\label{eq: variance bound}
            \E(|\nabla F(x_t)-G(x_t,\sigma,p)|^2) = \E\left(\E\left(|\nabla F(x_t)-G(x_t,\sigma,p)|^2 \big| x_t\right)\right)\leq 8\sigma^2 n^{-\frac{2(p-1)}{p}}
        \end{equation}for any $t=1,\ldots,T$.
         Plugging (\ref{eq: variance bound}) into (\ref{eq: telescoped sum}), we obtain that
    \begin{align*}
        \sum_{t=1}^T\eta_t\E(|\nabla F(x_t)|^2) \leq 2\Delta_1 + \frac{8\sigma^2}{n^{\frac{2(p-1)}{p}}} \sum_{t=1}^T\eta_t\;,
    \end{align*}
and dividing both sides by $\sum_{t=1}^T \eta_t$ concludes the proof.
\end{proof}

% Note that if in Theorem \ref{thm: in expectation convergence of SGD} we simply require $\alpha \geq \sigma$ instead of $\alpha = \sigma$ we obtain the following result.

% \begin{lemma}\label{lem: general in expectation convergence of SGD}
% Under Assumptions \ref{ass:assumptions for standard SGD} and for any $\eta_t < 1/L$, the iterates generated by PS-Clip-SGD with $\alpha \geq \sigma$ and $\beta = p$ satisfy
% \begin{align*}
% \sum_{t=1}^T\frac{\eta_t\E(|\nabla F(x_t)|^2)}{{\sum_{t=1}^T\eta_t}} \leq \frac{2\Delta_1}{\sum_{t=1}^T\eta_t} +  8\sigma^{p}\alpha^{2-p}\left((\sigma\alpha^{-1})^{p}+1\right)n^{-\frac{2(p-1)}{p}}\;.
% \end{align*}
% \end{lemma}
% \begin{proof}
%     Follows by the same exact steps as the proof of Theorem \ref{thm: in expectation convergence of SGD} and using the upper bound in Lemma \ref{lemm: generic convergence in espectation} instead of the one in Corollary \ref{cor: generic convergence in espectation with known parameters}.
% \end{proof}

\section{Proof of high-probability convergence}\label{Proof of high-probability convergence}
Similarly to the proof of the in-expectation convergence result, we begin by proving a generic high-probability error bound for the clipped mean estimator.
\begin{lemma}\label{lem: genereal high prob bound}
    Let $Y_1, \ldots, Y_n \sim Y$ be i.i.d. random variables such that $\E\left|Y\right|^{p}\leq \sigma^p<\infty$ for some $p\in (1,2]$. For any $\delta\in(0,1)$ and $\alpha\geq \sigma$ set 
    \begin{equation*}
        \gamma_k = \min\left\{1,\frac{\alpha k^{1/p}}{\log(1/\delta)^{1/p}|Y_k|}\right\}\;.
    \end{equation*}
    Then, with probability at least $1-\delta e^{1/4}$ we have that
    \begin{equation*}
        \left|\frac{1}{n}\sum_{k=1}^n Y_k \gamma_k - \E(Y)\right|\leq 7\alpha\left(\frac{\log(1/\delta)}{n}\right)^{\frac{p-1}{p}}\;.
    \end{equation*}
\end{lemma}
\begin{proof}
The proof of this result follows the same idea as the one of \citep[Lemma~1]{bandits}.
First, note that for every $k=1, \ldots,n$ we have that 
\begin{align*}
    &\mathbb{E}\left(Y_k\gamma_k-\E(Y_k\gamma_k)\right) = 0 \;,\\
&|Y_k\gamma_k-\E(Y_k\gamma_k)| \leq 2 \alpha\left(\frac{n}{\log(1/\delta)} \right)^{\frac{1}{p}}
\quad \text{and} \\ 
&\mathbb{E}\big[|Y_k \gamma_k-\E(Y_k\gamma_k)|^2\big] \leq \sigma^{p}\alpha^{2-p}\left(\frac{n}{\log(1/\delta)}\right)^\frac{2-p}{p}\;.
\end{align*}
Hence, using the second inequality in Lemma~\ref{lem: vector berstein inequality} we obtain that with probability $1-\delta e^{1/4}$
\begin{equation}\label{eq: generic high prob bound in proof}
    \begin{aligned}
\left|\frac{1}{n}\sum_{k=1}^n Y_k\gamma_k - \E(Y_k\gamma_k)\right|
&\leq
2 \alpha\left(\frac{n}{\log(1/\delta)} \right)^{\frac{1}{p}}\frac{ \log(1/\delta)}{n}+\sqrt{8\sigma^{p}\alpha^{2-p}\left(\frac{n}{\log(1/\delta)}\right)^\frac{2-p}{p}\frac{\log(1/\delta)}{n}}\\
&=\alpha\left(2+\sqrt{8}(\sigma\alpha^{-1})^{p/2}\right)\left(\frac{\log(1/\delta)}{n}\right)^{\frac{p-1}{p}}
<5\alpha\left(\frac{\log(1/\delta)}{n}\right)^{\frac{p-1}{p}}
\end{aligned}
\end{equation}
Combining this, with (\ref{eq: expectation bias}) we obtain that with probability $1-\delta e^{1/4}$
\begin{align*}
\left|\E(Y) 
- \frac{1}{n}\sum_{k=1}^n Y_k \gamma_k\right|
&\leq
\left|\frac{1}{n}\sum_{k=1}^n 
\Big(\E(Y_k) - \E(Y_k \gamma_k)\Big)\right|
+\left|
\frac{1}{n}\sum_{k=1}^n 
\Big(\E(Y_k\gamma_k)
- Y_k \gamma_k \Big)\right|
\\
&\leq
\frac{1}{n}\sum_{k=1}^n 
\sigma^p\alpha^{1-p}
\log(1/\delta)^{\frac{p-1}{p}} k^{\frac{1-p}{p}}
+5\alpha\left(\frac{\log(1/\delta)}{n}\right)^{\frac{p-1}{p}}
\\
&\leq
\frac{2}{n} 
\sigma^p\alpha^{1-p}\log(1/\delta)^{\frac{p-1}{p}} n^{\frac{1}{p}}
+5\alpha\left(\frac{\log(1/\delta)}{n}\right)^{\frac{p-1}{p}}
\\
&\leq
2\sigma^p\alpha^{1-p}\log(1/\delta)^{\frac{p-1}{p}} n^{\frac{1-p}{p}}
+5\alpha\left(\frac{\log(1/\delta)}{n}\right)^{\frac{p-1}{p}}
\\
&\leq 7\alpha\left(\frac{\log(1/\delta)}{n}\right)^{\frac{p-1}{p}}
\end{align*}
where in the second inequality we used (\ref{eq: expectation bias}) to estimate the first term and (\ref{eq: generic high prob bound in proof}) for
the second one.
\end{proof}
% \begin{corollary}
%     Let $\xi^{(1)},\ldots,\xi^{(n)} \sim \mathcal{D}$ be i.i.d samples, $\delta \in (0,1)$ and define
%     \begin{equation*}
%         \gamma_k = \min\left\{1,\frac{\sigma k^{1/p}}{\log(1/\delta)^{1/p}|\nabla(x,\xi^{(k)})|}\right\}\;.
%     \end{equation*}
%     Then, for any $x\in \R^d$ we have under Assumptions
%     \ref{ass:assumptions for standard SGD} that 
    
% \end{corollary}

% \begin{corollary}\label{cor: union bound}
%     Let $x_1,\ldots,x_T \in \R^d$ and $\delta \in (0,1)$. Then under the same assumptions of Lemma \ref{lem: genereal high prob bound} we have that
%     \begin{equation}
%         P\left(\max_{t=1,\ldots,T}\left|\frac{1}{n}\sum_{k=1}^n Y_k \gamma_k - \E(Y\gamma_k)\right|\leq 5\sigma\left(\frac{\log(T/\delta)}{n}\right)^{\frac{p-1}{p}}\right)\leq 1-\delta e^{1/4}
%     \end{equation}
% \end{corollary}
Setting $\alpha = \sigma$ in the previous Lemma we obtain the following corollary.

\begin{corollary}\label{cor: genereal high prob bound}
    Let $Y_1, \ldots, Y_n \sim Y$ be i.i.d. random variables such that $\E\left|Y\right|^{p}\leq \sigma^p<\infty$ for some $p\in (1,2]$, and for any $\delta\in(0,1)$  set 
    \begin{equation*}
        \gamma_k = \min\left\{1,\frac{\sigma k^{1/p}}{\log(1/\delta)^{1/p}|Y_k|}\right\}\;.
    \end{equation*}
    Then, with probability at least $1-\delta e^{1/4}$ we have that
    \begin{equation*}
        \left|\frac{1}{n}\sum_{k=1}^n Y_k \gamma_k - \E(Y)\right|\leq7\sigma\left(\frac{\log(1/\delta)}{n}\right)^{\frac{p-1}{p}}\;.
    \end{equation*}
\end{corollary}
We can now prove Theorem \ref{thm: high prob covergence of SGD}.
\begin{proof}[Proof of Theorem \ref{thm: high prob covergence of SGD}]
By Proposition \ref{prop: standard bound on iterates} we have that
    \begin{align}\label{eq: second telescoped sum}
            \sum_{m=1}^T\eta_t|\nabla F(x_t)|^2 &\leq 2\Delta_1+\sum_{m=1}^T\eta_t|\nabla F(x_t)-G(x_t,\alpha,p)|^2.
        \end{align}
Further,we know from Corollary \ref{cor: genereal high prob bound} that for every $x_t \in \R^d$,
\begin{equation*}
    |\nabla F(x)-G(x,\alpha,p)|^2\leq  49\sigma^2\left(\frac{\log(1/\delta)+1/4}{n}\right)^{\frac{2(p-1)}{p}}
\end{equation*}
with probability $1-\delta$. Hence, applying union bound, it follows immediately that with probability at least $1-\delta$,
\begin{equation}\label{eq: uniform error bound}
    \max_{t=1,\ldots,T}|\nabla F(x_t)-G(x_t,\alpha,p)|^2\leq  49\sigma^2\left(\frac{\log(T/\delta)+1/4}{n}\right)^{\frac{2(p-1)}{p}}\;.
\end{equation}
Plugging (\ref{eq: uniform error bound}) into (\ref{eq: second telescoped sum}) we obtain that with probability at least $1-\delta$,
\begin{align*}
            \sum_{m=1}^T\eta_t|\nabla F(x_t)|^2 &\leq 2\Delta_1+49\sigma^2\left(\frac{\log(T/\delta)+1/4}{n}\right)^{\frac{2(p-1)}{p}}\sum_{m=1}^T\eta_t
        \end{align*}
and the result follows by dividing both sides by $\sum_{t=1}^T \eta_t$.
\end{proof}

\section{Details of Experiments}
All experiments were conducted on an NVIDIA A100-SXM4-40GB GPU in an internal cluster.
\subsection{Quadratic function with noise}\label{app: Quadratic function with noise}
The following table provides additional details on the experiment presented in Figure \ref{fig:untuned convergence}.

\begin{table}[h]
  \caption{Performance comparison of Clip-SGD, PS-Clip-SGD and Normalized SGD in minimizing the function $f(x,\xi)=\frac{1}{2}|x|^2+\langle x,\xi\rangle$ with untuned hyperparameters.}
  \label{tab: performance comparison quadratic prob untuned params}
  \centering
  \begin{tabular}{clcc}
    \toprule
    Noise parameter    & Algorithm     & Minimal Gradient Norm & Average Gradient Norm \\
    \midrule
    \multirow{3}{*}{$p = 1.8$} 
        & Clip-SGD         & $0.0758$ & $0.4155$ \\
        & PS-Clip-SGD      & $\mathbf{0.0707}$ & $\mathbf{0.2269}$ \\
        & Normalized SGD   & $0.0763$ & $0.4156$ \\
    \midrule
    \multirow{3}{*}{$p = 1.5$} 
        & Clip-SGD         & $0.0940$ & $0.5661$ \\
        & PS-Clip-SGD      & $\mathbf{0.0937}$ & $\mathbf{0.2666}$ \\
        & Normalized SGD   & $0.0943$ & $0.5662$ \\
    \midrule
    \multirow{3}{*}{$p = 1.2$} 
        & Clip-SGD         & $0.2469$ & $1.0709$ \\
        & PS-Clip-SGD      & $\mathbf{0.1551}$ & $\mathbf{0.3285}$ \\
        & Normalized SGD   & $0.2469$ & $1.0709$ \\
    \bottomrule
  \end{tabular}
\end{table}

We also evaluated the performance of Clip-SGD, PS-Clip-SGD, and Normalized SGD after extensive hyperparameter tuning. Table \ref{tab: quadratic func tuned params} summarizes the optimal parameters for each algorithm, where $\eta$ denotes the step size, $\gamma$ is the clipping threshold for Clip-SGD, and $\alpha$ and $\beta$ are defined as in (\ref{eq: definition of estimator}). Note that for PS-Clip-SGD we only tuned $\alpha$, while setting $\beta = p$ as suggested by our theoretical results.

Figure~\ref{fig:tuned convergence} shows the average performance of the three algorithms over 10 runs using these tuned parameters. As in the untuned setting, we observe that PS-Clip-SGD outperforms both other methods. Moreover, consistent with the trends observed in Figure~\ref{fig:tuned convergence}, heavier-tailed noise adversely affects the convergence of Normalized SGD and Clip-SGD, whereas PS-Clip-SGD remains largely unaffected.

This behavior is further confirmed by the high-probability convergence experiment shown in Figure~\ref{fig: tuned high prob convergence}. As in the untuned setting, we run the algorithm $10^4$ times for $T = 100$ iterations and, for each run, compute the average gradient norm over the $T$ iterations. Figure~\ref{fig: tuned high prob convergence} shows the $(1 - \delta)$-quantiles of the average gradient norm plotted against $\log(1/\delta)$. As in the untuned case, PS-Clip-SGD clearly outperforms both Normalized SGD and Clip-SGD, confirming the improved dependence on $\log(1/\delta)$ in the upper bound for PS-Clip-SGD compared to the other methods, as discussed in Section~\ref{sec: high prob convergence}.

Comparing the results in this section with Figure \ref{fig: high prob convergence}, is also interesting to observe that, while Normalized SGD and Clip-SGD appear highly sensitive to the choice of hyperparameters, PS-Clip-SGD performs well with both tuned and untuned hyperparameters. This is again consistent with our theoretical analysis, which suggested that PS-Clip-SGD should be more robust to the choice of hyperparameters than the other two methods.
\begin{table}[!htbp]
  \caption{Tuned Hyperparameters for Clip-SGD, PS-Clip-SGD and Normalized SGD for minimizing the function $f(x,\xi)=\frac{1}{2}|x|^2+\langle x,\xi\rangle$.}
  \label{tab: quadratic func tuned params}
  \centering
  \begin{tabular}{clcccc}
    \toprule
    Noise parameter&Algorithm& $\eta$&$\gamma$ & $\alpha$ & $\beta$\\
    \midrule
    \multirow{3}{*}{$p = 1.8$} 
        & Clip-SGD         & $0.5$ & $0.1$& - & - \\
        & PS-Clip-SGD      & $0.05$ & - &1.0 & 1.8\\
        & Normalized SGD   & $0.05$ & - & - &-  \\
    \midrule
    \multirow{3}{*}{$p = 1.5$} 
        & Clip-SGD         & $0.05$ & $0.6$& - & - \\
        & PS-Clip-SGD      & $0.05$ & - &1.0 & 1.5\\
        & Normalized SGD   & $0.05$ & - & - &-  \\
    \midrule
    \multirow{3}{*}{$p = 1.2$} 
        & Clip-SGD         & $0.4$ & $0.1$& - & - \\
        & PS-Clip-SGD      & $0.05$ & - &1.0 & 1.2\\
        & Normalized SGD   & $0.05$ & - & - &-  \\
    \bottomrule
  \end{tabular}
\end{table}
\begin{figure}[!htbp]
    \centering
    \includegraphics[width=0.32\textwidth]{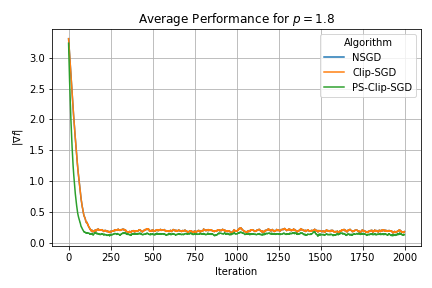}
    \hfill
    \includegraphics[width=0.32\textwidth]{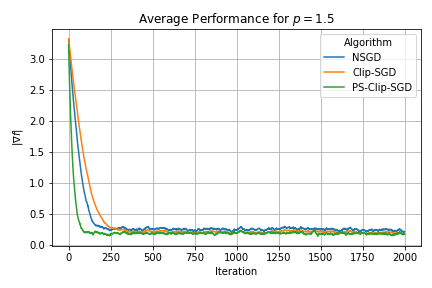}
    \hfill
    \includegraphics[width=0.32\textwidth]{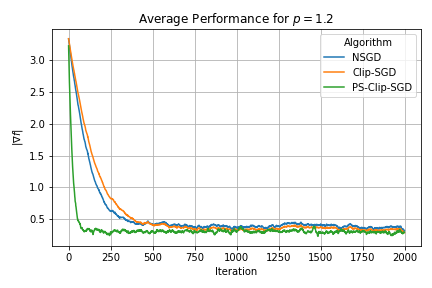}
    \caption{Performance of Normalized-SGD, Clip-SGD and PS-Clip-SGD for different noise regimes using the parameters in Table \ref{tab: quadratic func tuned params}. In the first plot the graphs of Normalized SGD and Clip-SGD appear indistinguishable.}
    \label{fig:tuned convergence}
\end{figure}
\begin{figure}[h]
    \centering
    \includegraphics[width=0.32\textwidth]{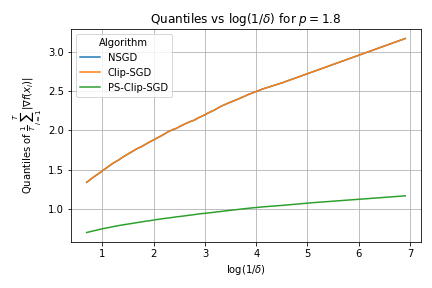}
    \hfill
    \includegraphics[width=0.32\textwidth]{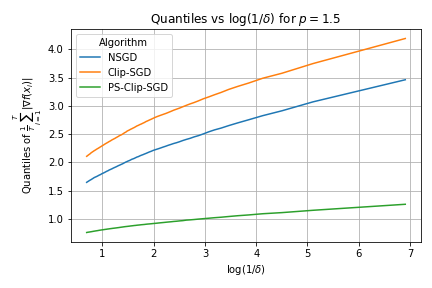}
    \hfill
    \includegraphics[width=0.32\textwidth]{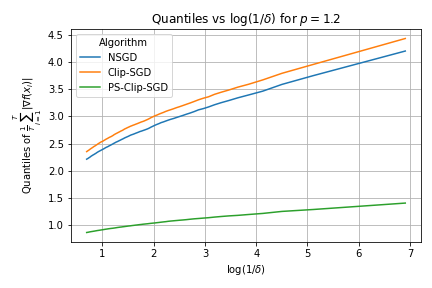}
    \caption{$(1-\delta)$-quantile of the average gradient norm after $T = 100$ training steps, plotted against $\log(1/\delta)$ for the three algorithms and different noise regimes. The experiment is performed using the tuned hyperparameters from Table \ref{tab: quadratic func tuned params}. In the first plot the graphs of Normalized SGD and Clip-SGD appear indistinguishable.}
    \label{fig: tuned high prob convergence}
\end{figure}
\newpage
\subsection{Additional details and experiments on training AlexNet with per-sample clipping}\label{app: alexnet}
\subsubsection{Model and Hyperparameters}
For the experiment presented in Section \ref{sec: Training AlexNet with per-sample clipping}, we used the AlexNet implementation from \texttt{torchvision.models} \footnote{\url{https://docs.pytorch.org/vision/main/models/generated/torchvision.models.alexnet.html}}. To make the network more suitable for the CIFAR-100 image size we replaced the first convolutional layer, \texttt{nn.Conv2d(3, 64, kernel\_size=11, stride=4, padding=2)}, with \texttt{nn.Conv2d(3, 64, kernel\_size=3, stride=1, padding=1)}. Additionally, we modified all \texttt{nn.MaxPool2d(kernel\_size=3, stride=2)} layers to \texttt{nn.MaxPool2d(kernel\_size=2, stride=2)}. We also replaced \texttt{nn.ReLU(inplace=True)} with \texttt{nn.ReLU(inplace=False)} to make the model compatible with Opacus. The rest of the network was left unchanged.

We randomly split the training portion of the CIFAR-100 dataset into 45,000 training images and 5,000 validation images. The training images were augmented using PyTorch's \texttt{AutoAugment} transform \footnote{\url{https://docs.pytorch.org/vision/main/generated/torchvision.transforms.AutoAugment.html\#torchvision.transforms.AutoAugment}} with the CIFAR-10 policy.  The momentum and weight decay parameters were set to $0.9$ and $10^{-4}$, respectively. The learning rate and clipping threshold were tuned over the grids $(0.1, 0.01, 0.001)$ and $(1.0, 15.0, 30.0, 45.0)$, respectively. Since the performance of Clip-SGD with parameter pairs $(0.01, 15)$ and $(0.01, 30)$ was similar during the tuning phase, we conducted full training runs with both configurations. The pair $(0.01, 15)$ performed slightly better, achieving a final validation accuracy of $58.94\%$, compared to $58.76\%$ for $(0.01, 30)$.

For the same reason, we also conducted full training runs for PS-Clip-SGD with both parameter pairs $(0.01, 30.0)$ and $(0.01, 45.0)$. The latter achieved a validation accuracy of $64.84\%$, while the pair $(0.01, 30.0)$ achieved $63.50\%$. Comparing these values with Table~\ref{tab: AlexNet training}, we note that PS-Clip-SGD still significantly outperforms both vanilla SGD and Clip-SGD even with the suboptimal choice of parameters $(0.01, 30.0)$.

\subsubsection{Increasing vs constant clipping threshold}\label{app: Increasing vs constant clipping threshold}

Since our theoretical analysis relies on a clipping threshold that increases with each sample, we also evaluated this strategy for training AlexNet on CIFAR-100. Our experiments show that, while this approach also outperforms both vanilla SGD and standard Clip-SGD, a constant clipping threshold may yield better performance in practice.

Recall that the clipping threshold for the $k$-th sample in a batch is defined in (\ref{eq: definition of estimator}) as $\alpha k^{1/\beta}$. We tested the parameters $\alpha$ and $\beta$ over the grids $(1.0, 15.0, 30.0,45.0)$ and $(1.2, 1.5, 1.8, 2.0)$, respectively, and compared the validation accuracy after $40$ epochs. As in our main experiment, we used a learning rate of $0.01$, momentum $0.9$, and weight decay with parameter $10^{-4}$.

Table~\ref{tab: performance after 40 epochs} summarizes the performance of the three best parameter pairs, namely $(15.0, 2.0)$, $(15.0, 1.8)$, and $(15.0, 1.5)$. We also report the performance of vanilla SGD, Clip-SGD, and PS-Clip-SGD with a constant clipping threshold, using their respective optimized hyperparameters. To distinguish between the two variants, in the table we refer to per-sample clipped SGD with an increasing clipping threshold as \emph{Increasing-PS-Clip-SGD-mom}, and to the version with a constant threshold as \emph{PS-Clip-SGD-mom}.

We observe that per-sample clipping with an increasing threshold outperforms both vanilla SGD and Clip-SGD for all three values of $\beta$. However, it still performs slightly worse than PS-Clip-SGD with a constant clipping threshold. Moreover, the performance of Increasing-PS-Clip-SGD improves as $\beta$ increases. This suggests that a constant per-sample clipping threshold may generally be preferable in practice, since larger values of $\beta$ imply a slower increase of the clipping threshold with $k$, with the limiting case ``$\beta = \infty$'' corresponding to a constant threshold.

\begin{table}[!htbp]
  \caption{Performance of SGD, Clip-SGD, PS-Clip-SGD, and Increasing-PS-Clip-SGD when training AlexNet on CIFAR-100 for $40$ epochs. All four algorithms use learning rate of $0.01$, momentum $0.9$, and weight decay with parameter $10^{-4}$.}
  \label{tab: performance after 40 epochs}
  \centering
  \begin{tabular}{clcccc}
    \toprule
    Learning rate &Algorithm&$\gamma$ & $\alpha$ & $\beta$ & Val loss after 40 epochs\\
    \midrule
    \multirow{6}{*}{$0.01$} 
        & SGD-mom & - & - & $-$ & $57.40 \%$\\
        & Clip-SGD-mom         & $15.0$ & - & - & $57.82\%$ \\
        & PS-Clip-SGD-mom      & $45.0$ & - &-  & $\mathbf{60.20\%}$\\
        & Increasing-PS-Clip-SGD-mom   & - & $15.0$ & $2.0$ &$59.42\%$ \\
        &Increasing-PS-Clip-SGD-mom   & - & $15.0$ & $1.8$ &$58.90\%$\\
        &Increasing-PS-Clip-SGD-mom  & - & $15.0$ & $1.5$ &$58.48\%$\\
    \midrule
    \bottomrule
  \end{tabular}
\end{table}

\subsubsection{Possible reason for improved performance}\label{sec: Possible reason for improved performance}
\begin{figure}[h]
    \centering
    \includegraphics[width=0.6\textwidth]{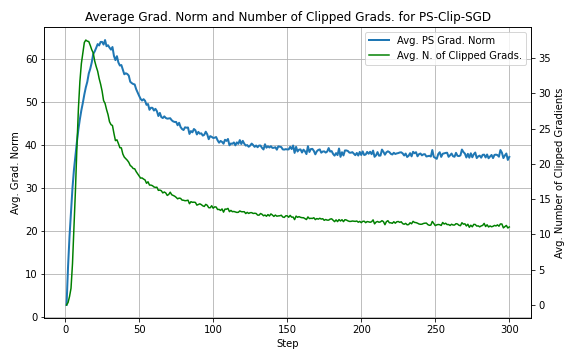}
    \hfill
    \caption{Blue line, left $y$-axis: Average per-sample gradient norm in each epoch: $\frac{1}{n_{batches}}\sum_{t=1}^{n_{batches}}\frac{1}{\textit{batch\_size}}\sum_{i=1}^{\textit{batch\_size}} |\nabla f(x_t,\xi_t^{(i)})|$. Green line, right $y$-axis: average number of clipped gradients in a batch for PS-Clip-SGD: $\frac{1}{n_{batches}}\sum_{t=1}^{n_{batches}}\frac{1}{\textit{batch\_size}}\sum_{i=1}^{\textit{batch\_size}} n_{clipped}^{(t,i)}$.\\
    Note: $n_{batches}$ indicates the total number of batches in an epoch, while $n_{clipped}^{(t,i)}$ indicates the number of clipped gradients in batch $i$ of epoch $t$.}
    \label{fig: AlexNet Gnorms}
\end{figure}
In this section, we provide a possible explanation for the improved performance of PS-Clip-SGD compared to vanilla SGD and Clip-SGD in the experiment presented in Section~\ref{sec: Training AlexNet with per-sample clipping}.

Figure~\ref{fig: AlexNet Gnorms} shows the average per-sample gradient norm and the average number of clipped gradients when training AlexNet with PS-Clip-SGD using a clipping threshold of $45.0$. We observe that the average number of clipped gradients per batch quickly decreases to values below $20$ and stabilizes around $11$. Recall that, in this experiment, we used a batch size of $64$. On the other hand, the average per-sample gradient norm remains consistently below $45.0$. This observation suggests the presence of a small number of outliers with very large gradients, which are effectively clipped by the algorithm, resulting in less noisy gradients and more stable parameter updates.

This also helps explain why standard (global) clipping is less effective in this setting than per-sample clipping. If, for instance, the norm of the batch gradient is consistently below $25.0$, then a global clipping threshold of $45.0$ or $30.0$ would have little effect. Moreover, the norm of the batch gradient,
$
\frac{1}{\textit{batch\_size}} \left| \sum_{i=1}^{\textit{batch\_size}} \nabla f(x_t, \xi_t^{(i)}) \right|,
$
can in general be significantly smaller than the mean per-sample gradient norm, $\frac{1}{\textit{batch\_size}} \sum_{i=1}^{\textit{batch\_size}} \left| \nabla f(x_t, \xi_t^{(i)}) \right|$.
This is precisely what we observed when training AlexNet with Clip-SGD and a clipping threshold of $15.0$: in this case, only about one gradient per epoch was clipped, so Clip-SGD behaved very similarly to vanilla SGD, as noted in Section~\ref{sec: Training AlexNet with per-sample clipping}. 
On the other hand, while an even lower clipping thresholds may prevent extreme movements in the parameter updates by controlling the gradient norm, the direction of the clipped gradient would still be noisy.

\subsection{Training GPT-2}\label{app:Training GPT-2}

For this experiment, we used the GPT-2 implementation from \citep{karpathy_nanogpt_2022}. We also used the API provided in the repository to download, split, and preprocess the OpenWebText dataset \citep{Gokaslan2019OpenWeb}. The code in \citep{karpathy_nanogpt_2022} was released under the MIT License, and the OpenWebText dataset was released under the CC0 license.

For the choice of hyperparameters, we followed those used in \citep{GPT3-paper} for the 124M-parameter model. In addition to the learning rate of $0.6 \cdot 10^{-4}$ used therein, we also tested maximum learning rates of $0.8 \cdot 10^{-4}$ and $10^{-3}$. The model was trained for 5000 training steps using the AdamW optimizer with betas $(0.9, 0.95)$, $\epsilon = 10^{-8}$, and weight decay of $0.1$. The learning rate was increased linearly during the first 500 training steps, after which we applied cosine learning rate decay down to a minimum value of $10\%$ of the maximum learning rate. We use gradient accumulation with 64 accumulation steps and a mini-batch size of 8, which corresponds to approximately $5 \cdot 10^{5}$ tokens per batch. Gradient clipping was applied with a threshold of 1.0 for both methods.

Table~\ref{tab: GPT-Training times} reports the total training time for each run. The additional time required by MB-Clip-SGD can be explained by the fact that the experiment was performed on a single GPU. After each accumulation step, the computed gradient is clipped and added to the previously accumulated gradients stored in a separate dictionary. \texttt{optimizer.zero\_grad()} is then called to reset the network gradients before the next accumulation step. After all accumulation steps are completed, the accumulated gradients are copied from the dictionary into the network before performing the optimization step. If gradient accumulation were parallelized across multiple GPUs, we expect the total computational time to be nearly identical for MB-Clip-SGD and Clip-SGD. 

\begin{table}[h]
  \caption{Total training time for the experiment presented in Section \ref{sec: Training GPT-2 with mini-batch clipping}}
  \label{tab: GPT-Training times}
  \centering
  \begin{tabular}{clcc}
    \toprule
    Learning Rate    & Algorithm  & Total Training Time (min) & \\
    \midrule
    \multirow{2}{*}{$6\cdot 10^{-4}$} 
        & Clip-SGD         & $1454.73$ \\
        & MB-Clip-SGD      & ${1473.83}$ \\
    \midrule
    \multirow{2}{*}{$8\cdot 10^{-4}$} 
        & Clip-SGD         & $1455.12$  \\
        & MB-Clip-SGD      & ${1473.96}$ \\
    \midrule
    \multirow{2}{*}{$1\cdot 10^{-3}$} 
        & Clip-SGD         & $1455.23$  \\
        & MB-Clip-SGD      & ${1473.56}$  \\
    \bottomrule
  \end{tabular}
\end{table}

%%%%%%%%%%%%%%%%%%%%%%%%%%%%%%%%%%%%%%%%%%%%%%%%%%%%%%%%%%%%

% \newpage
% \input{checklist.tex}

\end{document}